\documentclass[12pt]{article}
\usepackage{a4}
\usepackage{amsmath}
\usepackage{amsthm}
\usepackage{amsfonts}
\usepackage{amssymb}
\usepackage{cite}
\usepackage{hyperref}
\usepackage{epsfig}
\usepackage{colonequals}
\newtheorem{theorem}{Theorem}
\newtheorem{lemma}[theorem]{Lemma}
\newtheorem{observation}[theorem]{Observation}
\newtheorem{corollary}[theorem]{Corollary}
\usepackage{tikz}
\newcommand{\FF}{\mathcal{F}}
\begin{document}
\title{Triangle-free planar graphs with small independence number}
\author{%
     Zden\v{e}k Dvo\v{r}\'ak\thanks{Computer Science Institute (CSI) of Charles University,
           Malostransk{\'e} n{\'a}m{\v e}st{\'\i} 25, 118 00 Prague, 
           Czech Republic. E-mail: \protect\href{mailto:rakdver@iuuk.mff.cuni.cz}{\protect\nolinkurl{rakdver@iuuk.mff.cuni.cz}}.
           Supported by project GA14-19503S (Graph coloring and structure) of Czech Science Foundation.}
\and
Jordan Venters\thanks{
University of Warwick, Coventry CV4 7AL, UK.
E-mail: \protect\href{mailto:J.Venters@warwick.ac.uk}{\protect\nolinkurl{J.Venters@warwick.ac.uk}}.
The work of this author has received funding from the European Research Council (ERC) under the European Union's Horizon 2020 research and innovation
programme (grant agreement No 648509). This publication reflects only its authors' view; the European Research Council Executive Agency is
not responsible for any use that may be made of the information it contains.}
}
\date{\today}
\maketitle
\begin{abstract}
Since planar triangle-free graphs are $3$-colourable, such a graph with $n$ vertices has an independent
set of size at least $n/3$.  We prove that unless the graph contains a certain obstruction, its independence
number is at least $n/(3-\varepsilon)$ for some fixed $\varepsilon>0$.  We also provide a reduction rule
for this obstruction, which enables us to transform any plane triangle-free graph $G$ into a plane triangle-free graph $G'$
such that $\alpha(G')-|G'|/3=\alpha(G)-|G|/3$ and  $|G'|\le (\alpha(G)-|G|/3)/\varepsilon$.
We derive a number of algorithmic consequences
as well as a structural description of $n$-vertex plane triangle-free graphs whose independence number is close to $n/3$.
\end{abstract}

What is the smallest independence number a planar graph on $n$ vertices can have?  By Four Colour Theorem,
each such graph is $4$-colourable, and the largest colour class gives an independent set of size at least $n/4$.
On the other hand, there are infinitely many planar graphs for that this bound is tight.  In fact, it is
an intriguing open problem to describe such graphs, and we do not even know any polynomial-time algorithm
to decide whether an $n$-vertex planar graph has an independent set larger than $n/4$.

In this paper, we study an easier related problem regarding independent sets in planar triangle-free graphs.
By Gr\"{o}tzsch' theorem~\cite{grotzsch1959}, these graphs are $3$-colourable, and thus such a graph with $n$
vertices has an independent set of size at least $n/3$.  Unlike the general case, this bound is not
tight---Steinberg and Tovey~\cite{SteinbergTovey1993} proved the lower bound $(n+1)/3$, and gave an infinite
family of graphs for that this bound is tight.  Dvo\v{r}\'ak et al.~\cite{indsettight} improved this bound
to $(n+2)/3$ except for the graphs from this family.  Furthermore, Dvo\v{r}\'ak and Mnich~\cite{dmnich}
proved that there exists $\varepsilon>0$ such that each $n$-vertex plane triangle-free graph in that every
$4$-cycle bounds a face has an independent set of size at least $n/(3-\varepsilon)$.

Let us recall that $\alpha(G)$ denotes the size of the largest independent set in $G$.
Dvo\v{r}\'ak and Mnich~\cite{dmnich} moreover described an algorithm with time complexity
$2^{O(\sqrt{a})}n$ that for an $n$-vertex planar triangle-free graph $G$ and an
integer $a\ge 0$ decides whether $\alpha(G)\ge (n+a)/3$.  This algorithm is based on a
generalization of the last result of the previous paragraph, which we will state after introducing a
couple of definitions.

Let $G$ be a plane graph and let $H$ be a subgraph of $G$, whose drawing in the plane
is inherited from $G$.  Note that each vertex or edge of $G$ that is not contained in $H$
is drawn in some face of $H$; we say that the faces of $H$ in that some part of $G$ is drawn
are \emph{full}.  Equivalently, a face $f$ of $H$ is full if and only if $f$ is not a face of $G$.
The supergraph of $H$ with respect to that we define the fullness of a face of $H$ will always
be clear from the context.  By $|G|$, we mean the number of vertices of $G$.
A \emph{$k$-face} of a plane graph is a face homeomorphic to an open disk bounded by a cycle of length $k$.
We are now ready to state the result.
\begin{theorem}[Dvo\v{r}\'ak and Mnich~\cite{dmnich}]\label{thm-largenosep}
There exists a constant $\gamma>0$ as follows.  Let $G$ be a plane triangle-free graph and let $H$ be its
subgraph.  If every $4$-cycle in $H$ bounds a face and every full face of $H$ is a $4$-face, then $\alpha(G)\ge \frac{|G|+\gamma|H|}{3}$.
\end{theorem}
Hence, if $\alpha(G)\le (n+a)/3$, then $G$ contains no such subgraph $H$ with more than $a/\gamma$ vertices, and consequently
it is easy to see that the tree-width of $G$ is at most $O(\sqrt{a})$.  This gives the aforementioned algorithm
by using the standard dynamic programming approach to deal with the bounded tree-width graph.

Our main result is a more precise characterization of $n$-vertex plane triangle-free graphs with no independent set larger than
$(n+a)/3$.  To state the result, we need to give a few more definitions.
We construct a sequence of graphs $T_1$, $T_2$, \ldots, which we call \emph{Thomas-Walls graphs} (Thomas and Walls~\cite{tw-klein} proved that they are exactly
the $4$-critical graphs that can be drawn in the Klein bottle without contractible cycles of length at most $4$).
Let $T_1$ be equal to $K_4$.  For $k\ge 1$, let $u_1u_3$ be any edge of $T_k$ that belongs to two triangles and let $T_{k+1}$ be obtained from $T_k-u_1u_3$
by adding vertices $x$, $y$ and $z$ and edges $u_1x$, $u_3y$, $u_3z$, $xy$, $xz$, and $yz$.  The first few graphs of this sequence are drawn in Figure~\ref{fig-thomaswalls}.
\begin{figure}
\begin{center}
\includegraphics[scale=0.8]{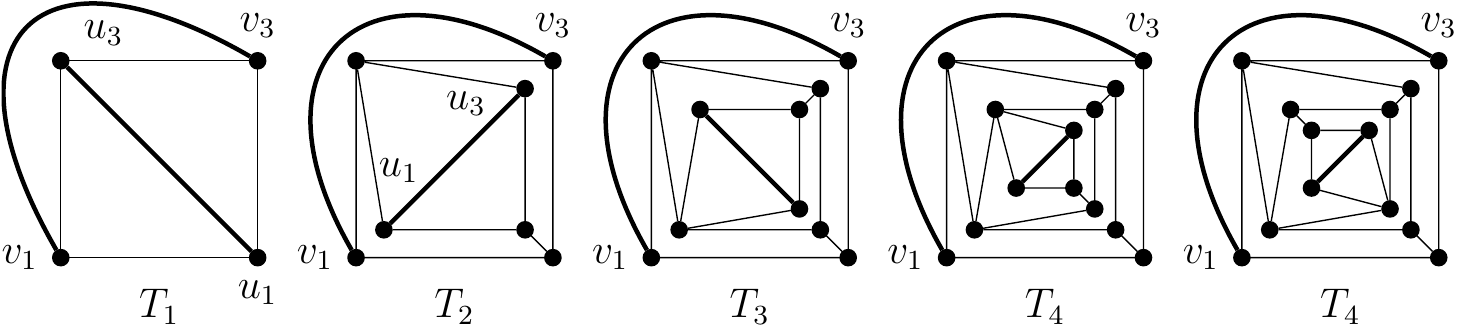}
\end{center}
\caption{Some Thomas-Walls graphs (the last two pictures demonstrate two combinatorially distinct drawings of $T_4$).}\label{fig-thomaswalls}
\end{figure}
For $k\ge 2$, note that $T_k$ contains unique $4$-cycles $C_1=u_1u_2u_3u_4$ and $C_2=v_1v_2v_3v_4$ such that $u_1u_3,v_1v_3\in E(G)$.
Let $T'_k=T_k-\{u_1u_3,v_1v_3\}$.  We also define $T'_1$ to be a $4$-cycle $C_1=C_2=u_1v_1u_3v_3$.
We call the graphs $T'_1$, $T'_2$, \ldots \emph{reduced Thomas-Walls graphs}, and we say that $u_1u_3$ and $v_1v_3$ are their \emph{interface pairs}.
Let us remark that the tight graphs found by Steinberg and Tovey~\cite{SteinbergTovey1993} are precisely those obtained from reduced Thomas-Walls graphs
by joining vertices of each interface pair by a path with three edges (and two new vertices).

Let $G$ be a plane triangle-free graph and let $H$ be a subgraph of $G$ isomorphic to a reduced
Thomas-Walls graph $T'_k$.  We say that $H$ is a \emph{clean Thomas-Walls $k$-tube} in $G$
if the $5$-faces of $H$ are not full.
Our characterization is based on the following strengthening of Theorem~\ref{thm-largenosep}.
\begin{theorem}\label{thm-main}
For every integer $k$, there exists a constant $\gamma_k>0$ as follows.
Every plane triangle-free graph without a clean Thomas-Walls $k$-tube satisfies $\alpha(G)\ge \frac{(1+\gamma_k)|G|}{3}$.
\end{theorem}
This is especially interesting in conjunction with the following observation.
Let $H$ be a clean Thomas-Walls $5$-tube in a plane triangle-free graph $G$, and let $C_1$ and $C_2$ be the cycles
bounding the $4$-faces of $H$.  Let $G'$ be the graph obtained from $G$ by removing the vertices of $V(H)\setminus V(C_1\cup C_2)$
and instead adding a copy of the reduced Thomas-Walls graph $T'_4$ with 4-faces bounded by cycles $C_1$ and $C_2$
and with the same interface pairs as $H$.  We say that $G'$ is a \emph{single-step TW-tube reduction} of $G$.
\begin{lemma}\label{lemma-redu}
Let $G$ be a plane triangle-free graph.  If $G'$ is a single-step TW-tube reduction of $G$, then $|G'|=|G|-3$ and
$\alpha(G)-|G|/3=\alpha(G')-|G'|/3$.
\end{lemma}
\begin{proof}
Let $H_0$ be the clean Thomas-Walls $5$-tube in $G$ transformed into a clean Thomas-Walls $4$-tube $H'_0$ in $G'$,
and let $\{u_2,u_4\}$ and $\{v_2,v_4\}$ be the vertices of their $4$-faces not belonging to the interface pairs.
Let $H$ and $H'$ be obtained from $H_0$ and $H'_0$ by
removing the vertices of their interface pairs.  Note that $Z=\{u_2,u_4,v_2,v_4\}$ is a cut in both $G$ and $G'$
separating $H$ or $H'$ from the rest of the graph.  For each $Z'\subseteq Z$ and $F\in\{H,H'\}$, let $\alpha_{Z'}(F)$
be the size of the largest independent set $S$ in $F$ such that $S\cap Z=Z'$.  A straightforward case analysis
shows that $\alpha_{Z'}(H)=\alpha_{Z'}(H')+1$ for every $Z'\subseteq Z$, and thus $\alpha(G)=\alpha(G')+1$.
Since $|G|=|G'|+3$, the claim of the lemma follows.
\end{proof}
Let $G''$ be obtained from $G$ by repeating a single-step TW-tube reduction as long as the graph contains a clean Thomas-Walls $5$-tube.
We say that $G''$ is the \emph{TW-tube reduction} of $G$. Note that $G''$ is uniquely determined by $G$ (as a graph---$G''$ may have
combinatorially different drawings in the plane), since a single-step TW-tube reduction cannot create a new Thomas-Walls tube,
only decrease the length of an existing one.
By Theorem~\ref{thm-main} and Lemma~\ref{lemma-redu}, we have the following.
\begin{corollary}\label{cor-tw-redu}
Let $G$ be a plane triangle-free graph and let $G'$ be its TW-tube reduction, and let $a=3\alpha(G)-|G|$.
Let $\gamma_5$ be the constant of Theorem~\ref{thm-main} applied with $k=5$.
Then $|G'|\le a/\gamma_5$ and $3\alpha(G')-|G'|=a$.
\end{corollary}
The argument showing the uniqueness of TW-tube reduction also gives the following
approximate characterization of plane triangle-free graphs whose independence number
is not much larger than one third of the number of vertices.
\begin{corollary}\label{cor-structure}
Let $\gamma_5$ be the constant of Theorem~\ref{thm-main} applied with $k=5$.
Let $G$ be a plane triangle-free graph and $a$ an integer.  If $\alpha(G)\le\frac{|G|+a}{3}$, then
$G$ has a subgraph $G'$ with at most $a/\gamma_5$ vertices such that each full face $f$ of $G'$
is bounded by two $4$-cycles and the subgraph of $G$ drawn in the closure of $f$ is a reduced Thomas-Walls graph.
Equivalently, $G$ can be obtained from a reduced Thomas-Walls graph by $O(a)$ additions and removals of vertices and edges.
\end{corollary}
Conversely, suppose $G$ has such a subgraph $G'$. Note that the number $q$ of full faces of $G'$ is bounded by the number of
components of $G'$ with at least $4$ vertices, and thus $q\le |G'|/4$.  Furthermore, observe that
$\alpha(T)\le \frac{|T|+5}{3}$ for every reduced Thomas-Walls graph $T$. The total number of vertices
of $G$ contained in the closures of full faces of $G'$ is at most $|G|-|G'|+8q$, and consequently,
$\alpha(G)\le |G'|+\frac{|G|-|G'|+13q}{3}<\frac{|G|+6|G'|}{3}$, and thus the structure from Corollary~\ref{cor-structure}
indeed approximately describes plane triangle-free graphs with independence number close to third of their number of vertices.

Corollary~\ref{cor-tw-redu} is also interesting from the algorithmic point of view.
An algorithmic problem is called \emph{fixed-parameter tractable} with respect to a parameter $p$ if
there exists a computable function $f$, a polynomial $q$, and an algorithm that solves each input instance $Z$ in time $f(p(Z))q(|Z|)$.
This notion has been influential in the area of computational complexity, giving a plausible approach towards many otherwise intractable
problems~\cite{downey2012parameterized,Niedermeier2006,cygan2015parameterized}.

A popular choice of the parameter is the value of the solution; i.e., such fixed-parameter tractability results show that
the solution to the problem can be found quickly if its value is small.  However, in the case of the problem of finding the largest independent set
when restricted to planar graphs, this parameterization makes little sense---the problem is fixed-parameter tractable
for the trivial (and unhelpful) reason that all large planar graphs in the class have large independent sets.
In this setting, parameterization by the excess of the size of the largest independent set over the lower bound
for the independence number is more reasonable.

The algorithm of Dvo\v{r}\'ak and Mnich~\cite{dmnich} with time complexity $2^{O(\sqrt{a})}n$
to decide whether an $n$-vertex planar triangle-free graph $G$ has an independent set of size at least $\frac{n+a}{3}$
thus shows that the largest independent set problem is fixed-parameter tractable in triangle-free graphs $G$ when
parameterized by $a=3\alpha(G)-|G|$.  Since the TW-tube reduction can be found in linear time, Corollary~\ref{cor-tw-redu}
can be interpreted as saying that the problem has \emph{linear-size kernel}, i.e., any instance can be reduced in time $O(n)$
to an equivalent instance of size $O(a)$.  As Robertson et al.~\cite{quickly} showed, an $m$-vertex planar graph has
tree-width $O(\sqrt{m})$, and thus its largest independent set can be found in time $2^{O(\sqrt{m})}$.
Consequently, we have the following.
\begin{corollary}
There exists an algorithm with time complexity $O(n)+2^{O(\sqrt{a})}$ that given an $n$-vertex plane triangle-free
graph and an integer $a\ge 0$ decides whether $\alpha(G)\ge \frac{n+a}{3}$.
\end{corollary}
Furthermore, Baker~\cite{baker1994approximation} gave an algorithm with time complexity $2^{O(1/\varepsilon)}m$
that given an $m$-vertex planar graph $G$ and a rational number $\varepsilon>0$ returns an independent set in $G$ of
size at least $(1-\varepsilon)\alpha(G)$.  Applying this algorithm to the TW-tube reduction, we obtain the following.
\begin{corollary}
There exists an algorithm with time complexity $O(n)+2^{O(1/\varepsilon)}$ that
given an $n$-vertex plane triangle-free
graph and a rational number $\varepsilon>0$ returns an independent set in $G$ of size at least
$\alpha(G)-\varepsilon(3\alpha(G)-|G|)$.
\end{corollary}
This improves over Baker's approximation for $n$-vertex plane triangle-free graphs whose independence number
is larger than $n/3$ by a sublinear additive factor.

In the rest of the paper, we present the proof of Theorem~\ref{thm-main}.
We start by giving some preliminary results in the next section, followed by
exploration of independent sets in graphs with almost all faces of length $4$ (Section~\ref{sec-4faces})
and in graphs containing a large reduced Thomas-Walls graph as a subgraph (Section~\ref{sec-tw}).
Finally, we derive the proof of the main result from previous structural results regarding
extensions of precolouring in triangle-free planar graphs, in Section~\ref{sec-main}.

\section{Preliminary results}

We will occasionally need to glue together independent sets from different subgraphs of a graph,
which we can do at the expense of removing the vertices in the intersection of the subgraphs from both
independent sets.
\begin{observation}\label{independenceadd}
Suppose $G$ is a plane graph and $H$ a subgraph of $G$.  Let $F$ be the subgraph
of $H$ induced by vertices incident with the full faces of $H$.
Given independent sets $I$ of $H$ and $J$ of $G-V(H)$, the set $(I\setminus V(F))\cup J$ is an independent
set of $G$.  Consequently,
$$\alpha(G)\ge \alpha(H)+\alpha(G-V(H))-\alpha(F)\ge \alpha(H)+\alpha(G-V(H))-|F|.$$
\end{observation}

Let $G$ be a graph and $c$ a positive integer.  By $[c]$, we mean the set $\{1,2,\ldots, c\}$.
A function $\varphi:V(G)\to 2^{[c]}$ is a \emph{colouring of $G$ by subsets of $[c]$} if $\varphi(u)\cap\varphi(v)=\emptyset$ for
all $uv\in E(G)$.  We say that $\varphi$ is a \emph{$(c:a)$-colouring} if $|\varphi(v)|=a$ for all $v\in V(G)$.  By a \emph{$c$-colouring}, we mean a $(c:1)$-colouring
(this matches the usual definition of a proper $c$-colouring); for brevity, we will write $\varphi(v)=x$ instead of $\varphi(v)=\{x\}$ for the integer $x\in [c]$ used to colour $v$.
Let us recall a well-known relation between colourings and independent sets.

\begin{lemma}\label{lemma-indset}
For any colouring $\varphi$ of a graph $G$ by subsets of $[c]$, we have
$$\alpha(G)\ge \frac{\sum_{v\in V(G)}|\varphi(v)|}{c}.$$
\end{lemma}
\begin{proof}
For any colour $i\in[c]$, the set $\{v\in V(G):i\in\varphi(v)\}$ is independent, and the sum of the sizes of these sets is $\sum_{v\in V(G)}|\varphi(v)|$.
Hence, at least one of the independent sets has the required size.
\end{proof}

Consider a plane triangle-free graph $G$ and let $v$ be a vertex of $G$ of degree at most $4$; let $v_1$, \ldots, $v_t$ with $t\le 4$ be neighbors
of $v$ in order around $v$ according to the drawing of $G$.  Let $G'$ be the graph obtained from $G-v$ by adding a $2t$-cycle $v_1v'_1v_2v'_2\ldots v_tv'_t$
with new vertices $v'_1$, \ldots, $v'_t$.  Note that $G'$ is plane triangle-free graph, and it follows from results of Gimbel and Thomassen~\cite{gimbel}
for $t\le 3$ and from the result of Dvo\v{r}\'ak and B. Lidick\'{y}~\cite{col8cyc} for $t=4$ that $G'$ has a $3$-colouring with vertices $v_1$, \ldots, $v_t$
coloured by colour $1$ and vertices $v'_1, \ldots, v'_t$ coloured by colour $2$.  Hence, there exists a $3$-colouring of $G-v$ in that all neighbors of $v$
have colour $1$, and thus the following is true.
\begin{lemma}\label{lemma-mono}
If $G$ is a planar triangle-free graph and a vertex $v\in V(G)$ has degree at most $4$, then there exist a colouring $\varphi$ of $G$ by subsets of $[3]$
such that $|\varphi(u)|=1$ for all $u\in V(G)\setminus \{v\}$ and $|\varphi(v)|=2$.
\end{lemma}

As observed by Steinberg and Tovey~\cite{SteinbergTovey1993}, together with Lemma~\ref{lemma-indset} and the fact
that every planar triangle-free graph has at least one vertex of degree at most three, this implies the following.
\begin{corollary}\label{cor-add1}
Every planar triangle-free graph $G$ has an independent set of size at least $\frac{|G|+1}{3}$.
\end{corollary}

We will need a umber of results concerning a precolouring extension in triangle-free planar graphs.
\begin{theorem}[Aksionov~\cite{aksenov}]\label{thm-extend}
Let $G$ be a plane triangle-free graph with the outer face bounded by a cycle $K$ of length at most $5$.
Every $3$-colouring of $K$ extends to a $3$-colouring of $G$.
\end{theorem}

\begin{theorem}[Gimbel and Thomassen~\cite{gimbel}]\label{thm-extend6}
Let $G$ be a triangle-free plane graph with the outer face bounded by a cycle $K$ of length $6$.
If a $3$-colouring of $K$ does not extend to a $3$-colouring of $G$, then $G$ contains a subgraph $G'$ with
outer face bounded by $K$, such that all other faces of $G'$ have length $4$.
\end{theorem}

\begin{corollary}\label{cor-patches}
Suppose $G$ is a triangle-free plane graph with the outer face bounded by a $6$-cycle $K=v_1v_2\ldots v_6$,
with a subgraph $H$ containing $K$ such that every other face of $H$ has length $4$.  Let $S=\{v_1,v_3,v_5\}$.
A $3$-colouring of $S$ extends to a $3$-colouring of $G$ if and only if some two vertices of $S$ have the same colour.
\end{corollary}

Furthermore, we need some results on precolouring extension in planar graphs with exactly one triangle.

\begin{theorem}[Aksionov~\cite{aksenov}]\label{thm-extend43}
Let $G$ be a plane graph with the outer face bounded by a cycle $K$ of length $4$.
If $G$ contains exactly one triangle, then every $3$-colouring of $K$ extends to a $3$-colouring of $G$.
\end{theorem}

\begin{theorem}[Dvo\v{r}\'ak et al.~\cite{trfree1}]\label{thm-extend63}
Let $G$ be a plane graph with the outer face bounded by a $6$-cycle $K$.  Suppose that $G$ contains exactly one triangle $T$
and all other cycles in $G$ have length at least $5$.  If a $3$-colouring of $K$ does not extend to a $3$-colouring
of $G$, then either $E(T)\setminus E(K)$ consists of a chord of $K$, or $T$ is vertex-disjoint from $K$ and each vertex
of $T$ has a neighbor in $K$.
\end{theorem}

\begin{corollary}\label{cor-extend63}
Let $G$ be a plane graph with the outer face bounded by a $6$-cycle $K$.  Suppose that $G$ contains exactly one triangle $T$
and there exists an edge $e_0\in E(T)$ such that every $4$-cycle in $G$ contains $e_0$.
Suppose that $K$ is an induced cycle, $T$ is vertex-disjoint from $K$, and at least one vertex of $T$ has no neighbor in $K$.
Then every $3$-colouring of $K$ extends to a $3$-colouring of $G$.
\end{corollary}
\begin{proof}
Without loss of generality, we can assume that
\begin{itemize}
\item[($\star$)] \emph{for every cycle $C$ in $G$, there exists a $3$-colouring of $C$ that does not extend to a $3$-colouring of
the subgraph of $G$ drawn in the closed disk bounded by $C$.}
\end{itemize}
Indeed, otherwise we can consider the subgraph of $G$ obtained by removing vertices and edges contained in the open disk bounded by $C$.

By Theorems~\ref{thm-extend}, \ref{thm-extend6}, \ref{thm-extend43}, and \ref{thm-extend63}, it follows that $T=v_1v_2y$ bounds a face of $G$,
and that $G$ contains exactly one $4$-cycle $v_1v_2v_3v_4$, also bounding a face.  Since $T$ is the only triangle in $G$,
the planarity implies that either $v_1yv_2v_3$ is the only path of length three between $v_1$ and $v_3$, or $v_2yv_1v_4$
is the only path of length three between $v_2$ and $v_4$; by symmetry, assume the former.

Let $G'$ be the graph obtained from $G$ by identifying $v_1$ with $v_3$ to a new vertex $v_{13}$ and suppressing the parallel edges.
Note that $T$ is the only triangle in $G'$.
If $G'$ contains a $4$-cycle $C'$, then $G$ contains a path $P$ of length $4$ between $v_1$ and $v_3$.  Let $C$ be the cycle consisting of $P$ and $v_1v_2v_3$.
Note that either $v_4$ or the face bounded by $T$ is contained in the open disk bounded by $C$.  The former is excluded by
Theorem~\ref{thm-extend6} and ($\star$).  In the latter case, the face bounded by $T$ in $G'$ is contained in the open disk bounded by $C'$,
and by Theorems~\ref{thm-extend6} and \ref{thm-extend43}, we conclude that every $3$-colouring of $K$ extends to a $3$-colouring of $G'$, and thus also to a $3$-colouring of $G$.

Hence, we can assume that $G'$ does not contain any $4$-cycles. 
Let $K=x_1x_2\ldots x_6$.
If some $3$-colouring of $K$ does not extend to a $3$-colouring of $G'$, then since $y\not\in V(K)$,
Theorem~\ref{thm-extend63} implies that $y$, $v_2$, and $v_{13}$ have neighbors in $K$, say $x_1$, $x_3$, and $x_5$, respectively.  Since not all vertices of $T$ have a neighbor in $K$ in the graph $G$,
we conclude that $v_3x_5\in E(G)$.  By planarity, we conclude that $v_2yv_1v_4$ is the only path of length three between $v_2$ and $v_4$.  Hence, we can consider the graph $G''$
obtained from $G$ by identifying $v_2$ with $v_4$.  Observe using Theorem~\ref{thm-extend63} that every $3$-colouring of $K$ extends to a $3$-colouring of $G''$, and thus also to a $3$-colouring of $G$.
\end{proof}

As we already mentioned, every planar triangle-free graph has a vertex of degree at most three.
Actually, Euler's formula implies the following stronger claim.

\begin{observation}\label{obs-three}
Let $G$ be a plane triangle-free graph with the outer face bounded by a cycle $K$ of length at most $5$.
If $G\neq K$, then $G$ has a vertex of degree at most $3$ not belonging to $K$.
\end{observation}

We will need another variation on the same theme.
\begin{lemma}\label{lemma-three-no2}
Let $G$ be a plane triangle-free graph with the outer face bounded by a cycle $K$ of length $5$ and let $uv$ be an
edge of $K$.  If $G\neq K$ and all vertices of $V(G)\setminus V(K)$ have degree at least three, then
$G$ has a vertex $z$ of degree three not belonging to $K$ such that $uz,vz\not\in E(G)$.
\end{lemma}
\begin{proof}
Without loss of generality, we can assume that $G$ is connected, otherwise we can find $z$ in a component of $G$
that does not contain $K$.
Give each vertex $x$ of $G$ charge $\deg(x)-4$ and each face $f$ charge $|f|-4$.

Each face has non-negative charge and the outer face has charge $1$.  Also, $G$ has at least one other odd-length
face with charge at least $1$.  If at least one vertex of $V(K)\setminus \{u,v\}$ has degree greater than two,
then the sum of their charges is at least $-5$.
If all vertices of $V(K)\setminus \{u,v\}$ have degree $2$, then the sum of their charges is $-6$ and
the incident non-outer face has length at least $7$ and charge at least $3$.
Consequently, sum of the charges of all faces of $G$ and of the vertices
of $V(K)\setminus \{u,v\}$ is at least $-3$.

Suppose for a contradiction that all non-neighbors of $u$ and $v$ in $V(G)\setminus V(K)$
have degree at least $4$, and thus their charge is non-negative.  Since vertices of $V(G)\setminus V(K)$ have degree at least three,
each neighbor of $u$ and $v$ in $V(G)\setminus V(K)$ has charge at least $-1$.
Consequently, the sum of all charges is at least $-3+(\deg(u)-4)+(\deg(v)-4)-(\deg(u)-2)-(\deg(v)-2)=-7$.
However, by Euler's formula, the sum of all charges is $-8$, which is a contradiction.
\end{proof}
\section{Graphs with many $4$-faces}\label{sec-4faces}

Let $\varphi$ be a $(3m:m)$-colouring of a $4$-cycle $v_1v_2v_3v_4$.  The \emph{margin} of $\varphi$
is the minimum of $|\varphi(v_1)\setminus\varphi(v_3)|$, $|\varphi(v_2)\setminus\varphi(v_4)|$, and $3m-\Bigl|\bigcup_{i=1}^4\varphi(v_i)\Bigr|$.
We will show that every $(3m:m)$-colouring of a $4$-cycle can be expressed as a union of $m$ $3$-colourings.
Note that up to permutation of colours, there are three possible $3$-colourings of a $4$-cycle, and we will also
show that if the $(3m:m)$-colouring has large margin, then the union contains many copies of each of the three
possible $3$-colourings.  In the statement of the lemma, the three $3$-colourings in question assign vertices of the $4$-cycle
colours $(A_1,A_2,A_1,A_3)$, $(B_1,B_2,B_3,B_2)$, and $(C_1,C_2,C_1,C_2)$ in order.

\begin{lemma}\label{deltagood}
Let $\varphi$ be a $(3m:m)$-colouring of a $4$-cycle $v_1v_2v_3v_4$
with margin $b$.  There exist pairwise disjoint sets $A_1,A_2,A_3, B_1, B_2, B_3, C_1,C_2,C_3\subset [3m]$ such that
\begin{align*}
|A_1|&=|A_2|=|A_3|\ge b\\
|B_1|&=|B_2|=|B_3|\ge b\\
|C_1|&=|C_2|=|C_3|\ge b\\
\varphi(v_1)&=A_1\cup B_1\cup C_1\\
\varphi(v_2)&=A_2\cup B_2\cup C_2\\
\varphi(v_3)&=A_1\cup B_3\cup C_1\\
\varphi(v_4)&=A_3\cup B_2\cup C_2.
\end{align*}
\end{lemma}
\begin{proof}
Let $S=\bigcup_{i=1}^4\varphi(v_i)$ and $C_3=[3m]\setminus S$; we have $|C_3|\ge b$.
Since $\varphi$ is a colouring of the $4$-cycle, the sets $\varphi(v_1)\cup \varphi(v_3)$ and $\varphi(v_2)\cup \varphi(v_4)$
are disjoint.  We will select $A_1$, $B_1$, $B_3$, and $C_1$ as pairwise disjoint subsets of $\varphi(v_1)\cup \varphi(v_3)$, and
$A_2$, $A_3$, $B_2$, and $C_2$ as pairwise disjoint subsets of $\varphi(v_2)\cup \varphi(v_4)$, ensuring that all nine sets
are pairwise disjoint.

Let $A_2=\varphi(v_2)\setminus\varphi(v_4)$ and $A_3=\varphi(v_4)\setminus \varphi(v_2)$.  Since $|\varphi(v_2)|=|\varphi(v_4)|=m$,
we have $|A_2|=|A_4|\ge b$.  Symmetrically, let $B_1=\varphi(v_1)\setminus \varphi(v_3)$ and
$B_3=\varphi(v_3)\setminus \varphi(v_1)$ and note that $|B_1|=|B_3|\ge b$.

Note that $|\varphi(v_1)\cap \varphi(v_3)|=2m-|\varphi(v_1)\cup\varphi(v_3)|=2m+|\varphi(v_2)\cup\varphi(v_4)|-|S|=3m+|A_3|-|S|=|A_3|+|C_3|$.
Hence, we can choose disjoint sets $A_1$ and $C_1$ such that $|A_1|=|A_3|$, $|C_1|=|C_3|$ and $A_1\cup C_1=\varphi(v_1)\cap \varphi(v_3)$.
Symmetrically, we can choose disjoint sets $B_2$ and $C_2$ such that $|B_2|=|B_3|$, $|C_2|=|C_3|$ and $B_2\cup C_2=\varphi(v_2)\cap \varphi(v_4)$.
It is easy to check these sets satisfy all of the requirements.
\end{proof}

\begin{corollary}\label{cor-extend}
Let $G$ be a plane triangle-free graph with the outer face bounded by a cycle $K$ of length $4$.
For every integer $m\ge 1$, every $(3m:m)$-colouring of $K$ extends to a $(3m:m)$-colouring of $G$.
\end{corollary}
\begin{proof}
Let $\psi$ be a $(3m:m)$-colouring of $K$, and let $A_1,\ldots, C_3$ be the sets from Lemma~\ref{deltagood}
applied to $\psi$.  Let $\psi_1$, $\psi_2$, and $\psi_3$ be the $3$-colourings of $K$ such that
\begin{align*}
(\psi_1(v_1), \psi_1(v_2), \psi_1(v_3), \psi_1(v_4))&=(1,2,1,3)\\
(\psi_2(v_1), \psi_2(v_2), \psi_3(v_3), \psi_2(v_4))&=(1,2,3,2)\\
(\psi_3(v_1), \psi_3(v_2), \psi_4(v_3), \psi_3(v_4))&=(1,2,1,2)
\end{align*}
For $i\in\{1,2,3\}$, Theorem~\ref{thm-extend} implies that $\psi_i$ extends to a $3$-colouring $\varphi_i$
of $G$.  For $u\in V(G)$, let us define
$\varphi(u)=A_{\varphi_1(u)}\cup B_{\varphi_2(u)}\cup C_{\varphi_3(u)}$.
Observe that $\varphi$ is a $(3m:m)$-colouring of $G$ that extends $\psi$.
\end{proof}

Similarly, if a $4$-cycle in a planar triangle-free graph has a $(3m:m)$-colouring with large margin, we can use
Lemma~\ref{lemma-mono} and extend the precolouring so that one vertex is assigned more than $m$ colours.

\begin{lemma}\label{deltaextend}
Suppose $G$ is a triangle-free plane graph with the outer face bounded by a $4$-cycle $K=v_1v_2v_3v_4$.
Let $\psi$ be a $(3m:m)$-colouring of $K$ and let $v$ be a vertex of $G$ of degree at most $3$ not belonging to $K$.
If $\psi$ has margin at least $b$, then $\psi$ extends to a colouring $\varphi$ of $G$ by subsets of $[3m]$ such that
$|\varphi(u)|=m$ for $u\in V(G)\setminus\{v\}$ and $|\varphi(v)|\ge m+b$.
\end{lemma}
\begin{proof}
Let $\psi_1$, $\psi_2$, and $\psi_3$ be the $3$-colourings of $K$ such that
\begin{align*}
(\psi_1(v_1), \psi_1(v_2), \psi_1(v_3), \psi_1(v_4))&=(1,2,1,3)\\
(\psi_2(v_1), \psi_2(v_2), \psi_3(v_3), \psi_2(v_4))&=(1,2,3,2)\\
(\psi_3(v_1), \psi_3(v_2), \psi_4(v_3), \psi_3(v_4))&=(1,2,1,2)
\end{align*}
Note that any $3$-colouring of $K$ can be obtained from one of these colourings by a permutation of colours.
Hence, using Lemma~\ref{lemma-mono} and Theorem~\ref{thm-extend}, we conclude that there exist colourings $\varphi_1$, $\varphi_2$, and $\varphi_3$
of $G$ by subsets of $[3]$ that extend $\psi_1$, $\psi_2$, and $\psi_3$, respectively, such that $|\varphi_i(u)|=1$ for $u\in V(G)\setminus\{v\}$ and $i\in \{1,2,3\}$,
$|\varphi_a(v)|=2$ for some $a\in\{1,2,3\}$, and $|\varphi_i(v)|=1$ for $i\in \{1,2,3\}\setminus\{a\}$.

Let $A_1$, \ldots, $C_3$ be the sets obtained by applying Lemma~\ref{deltagood} to $\psi$.  For $u\in V(G)$,
let us define
$$\varphi(u)=\bigcup_{c\in\varphi_1(u)}A_c\cup \bigcup_{c\in\varphi_2(u)}B_c\cup \bigcup_{c\in\varphi_3(u)}C_c.$$
Observe that $\varphi$ has the required properties.
\end{proof}

Next, we show that a $(3m:a)$-colouring with $a>m$ can be transformed into a $(3m:m)$-colouring that has some
margin on boundaries of $4$-faces.
A graph is \emph{$1$-planar} if it can be drawn in plane so that crossings between edges are allowed,
but each edge is crossed at most once.  

\begin{lemma}\label{deltagoodreduction}
Suppose $G$ is a triangle-free plane graph with a $(3m:m+8b)$-colouring $\varphi'$
for a non-negative integer $b$.  Then $G$ has a $(3m:m)$-colouring $\varphi$
such that $\varphi(v)\subseteq \varphi'(v)$ for all $v\in V(G)$, and for every $4$-face $f$ of $G$,
the restriction of $\varphi$ to the boundary cycle of $f$ has margin at least $b$.
\end{lemma}
\begin{proof}
Let $G'$ be the graph obtained from $G$ by adding both chords joining opposite vertices of every $4$-face of $G$.
Note that $G'$ is $1$-planar and hence $7$-degenerate~\cite{fabrici2007structure}, so we may label the vertices of $G'$
as $v_1,\ldots,v_n$ so that for every $i\in\{1,\ldots,n\}$, $v_i$ has at most seven neighbours in $\{v_1,\ldots,v_{i-1}\}$. 

Let $s_0$ be an arbitrary subset of $[3m]$ of size $b$.
For $i=1,\ldots, n$ in order, let $s_i$ be an arbitrary subset of size $b$ of the set
$$S_i\colonequals \varphi'(v_i)\setminus\Bigl(s_0\cup\bigcup_{j<i,v_jv_i\in E(G')} s_j\Bigr).$$
By the choice of the ordering of the vertices, we have $|S_i|\ge |\varphi'(v_i)|-8b\ge m$ for $i\in\{1,\ldots,n\}$.
Choose $\varphi(v_i)$ as an arbitrary subset of $S_i$ of size $m$ that contains $s_i$ as a subset.

Let $u_1u_2u_3u_4$ be a $4$-cycle bounding a face $f$ of $G$.  Note that $\bigcup_{j=1}^4 \varphi(u_j)\subseteq [3m]\setminus s_0$,
and thus $\Bigl|\bigcup_{j=1}^4 \varphi(u_j)\Bigr|\le 3m-b$.  Furthermore, let $i_1$ and $i_3$ be indices such that
$u_1=v_{i_1}$ and $u_3=v_{i_3}$, where by symmetry we can assume that $i_1<i_3$.  Since $u_1u_3$ is an edge of $G'$, we
have $s_{i_1}\subseteq \varphi(u_1)\setminus \varphi(u_3)$,
and thus $|\varphi(u_1)\setminus \varphi(u_3)|\ge b$.  Symmetrically, we have $|\varphi(u_2)\setminus \varphi(u_4)|\ge b$.
Consequently, the restriction of $\varphi$ to the boundary cycle of $f$ has margin at least $b$.
\end{proof}

Combining these results, we have the following.

\begin{lemma}\label{indepset}
Suppose	$G$ is a plane triangle-free graph and $H$ is a subgraph of $G$,
and let $t$ be the number of vertices of $H$ such that not all incident faces of $H$ are $4$-faces.
If $H$ has a $(3m:m+8b)$-colouring $\psi$ for a non-negative integer $b$,
then $G$ has an independent set of size at least
$$\frac{|G|+\varepsilon|H|-3t}{3},$$
where $\varepsilon=\frac{b}{4m}$.
\end{lemma}
\begin{proof}
Let $c$ be the number of full $4$-faces of $H$, and let $H'$ be the subgraph of $G$ obtained from $H$ by
adding all vertices and edges of $G$ drawn in the full $4$-faces of $H$.

By Lemma~\ref{deltagoodreduction}, there exists a $(3m:m)$-colouring $\varphi_0$ of $H$ whose restriction
to the boundary cycle of any $4$-face of $H$ has margin at least $b$, such that $\varphi_0(v)\subseteq \psi(v)$
for all $v\in V(H)$.  Let $\varphi_1$ be a set colouring such that $\varphi_1(v)=\varphi_0(v)$ for each $v\in V(H)$
incident with a full $4$-face of $H$, and $\varphi_1(v)=\psi(v)$ for all other vertices $v\in V(H)$.
By Corollary~\ref{cor-extend} and Lemma~\ref{deltaextend}, we can extend $\varphi_1$ to a colouring $\varphi$ of $H'$
by subsets of $[3m]$ such that all vertices are assigned sets of size at least $m$, vertices of $H$ not incident
with full $4$-faces are assigned sets of size $m+8b$, and $c$ vertices in $V(H')\setminus V(H)$ are assigned sets of size at least $m+b$.
By Lemma~\ref{lemma-indset}, we have
\begin{align*}
\alpha(H')&\ge \frac{|H'|m+bc+8b\max(0,|H|-4c)}{3m}\\
&=\frac{|H'|m+b\max(c,8|H|-31c)}{3m}\\
&\ge \frac{|H'|m+b|H|/4}{3m}=\frac{|H'|+\varepsilon|H|}{3}.
\end{align*}
Since every planar triangle-free graph is $3$-colourable, we have $\alpha(G-V(H'))\ge \frac{|G|-|H'|}{3}$.
Hence, Observation~\ref{independenceadd} implies
$$\alpha(G)\ge \frac{|G|+\varepsilon|H|}{3}-t.$$
\end{proof}

To apply this lemma, we need to find a $(3m:a)$-colouring with $a>m$ in a given graph.
This is possible based on the following result.
\begin{theorem}[Dvo\v{r}\'ak et al.~\cite{frpltr}]\label{thm-frpltr}
Let $n\ge 1$ be an integer.
If $G$ is a planar triangle-free graph with at most $n$ vertices, then $G$ has a $(9n : 3n + 1)$-colouring.
\end{theorem}

\begin{corollary}\label{cor-frpltr}
Let $t\ge 1$ be an integer.
If $G$ is a plane triangle-free graph and $G$ contains at most $t$ vertices such that not all incident faces of $G$ are $4$-faces,
then $G$ has a $(9t : 3t + 1)$-colouring.
\end{corollary}
\begin{proof}
We prove the claim by the induction on the number of vertices of $G$.  If $G$ has no $4$-faces, then $|G|\le t$ and the claim
follows from Theorem~\ref{thm-frpltr}.  Otherwise, let $v_1v_2v_3v_4$ be a $4$-face of $G$.  Since $G$ is triangle-free and plane,
it cannot contain both a path of length three between $v_1$ and $v_3$, and a path of length three between $v_2$ and $v_4$.
By symmetry, assume that $G$ does not contain a path of length three between $v_1$ and $v_3$, and thus the graph $G'$ obtained
from $G$ by identifying $v_1$ with $v_3$ and suppressing the bigon faces (but keeping other possible parallel edges if they arise,
so that no new faces are created)
is triangle-free.  By the induction hypothesis, $G'$ has a $(9t : 3t + 1)$-colouring, and thus $G$ also has a $(9t : 3t + 1)$-colouring
obtained by assigning $v_1$ and $v_3$ the same set of colours.
\end{proof}

This enables us to obtain a large independent set if we are given a subgraph in that almost all faces have length $4$.

\begin{lemma}\label{lemma-sizeind}
Suppose	$G$ is a plane triangle-free graph and $H$ is a subgraph of $G$,
and let $t$ be the number of vertices of $H$ such that not all incident faces of $H$ are $4$-faces.
Then $G$ has an independent set of size at least
$$\frac{|G|+|H|/(96t)-3t}{3}$$ if $t>1$ and of size at least
$$\frac{|G|+|H|/64}{3}$$ if $t=0$.
\end{lemma}
\begin{proof}
If $t>0$, then by Corollary~\ref{cor-frpltr}, $H$ has a $(9t : 3t + 1)$-colouring.  By replacing each colour with $8$ new colours,
we can obtain a $(72t:24t+8)$-colouring of $H$.  We apply Lemma~\ref{indepset} with $m=24t$ and $b=1$,
obtaining an independent set of size at least
$$\frac{|G|+|H|/(96t)-3t}{3}.$$
If $t=0$, then $H$ is bipartite, and thus $H$ has a $(48:24)$-colouring.  We apply Lemma~\ref{indepset} with $m=16$ and $b=1$,
obtaining an independent set of size at least
$$\frac{|G|+|H|/64}{3}.$$
\end{proof}

\section{Thomas-Walls graphs}\label{sec-tw}

In this section, we deal with the case that the considered graph contains a large reduced Thomas-Walls graph
as a subgraph, but many of its $5$-faces are full.  Also, we will deal with a derived class of subgraphs, so called
``patched Thomas-Walls graphs''.

Let us start with an observation regarding $3$-colourings of reduced Thomas-Walls graphs.
In any $3$-colouring of a $5$-cycle $K$, there is a unique vertex whose colour appears on $K$ exactly once;
we say that this vertex is the \emph{pivot} of the $3$-colouring on $K$.  Any two $3$-colourings of $K$ with the same
pivot differ only by a permutation of colours.

\begin{lemma}\label{lemma-coltw}
A reduced Thomas-Walls graph $H$ has four $3$-colourings $\varphi_1$, \ldots, $\varphi_4$ such that for each face
of $H$ bounded by a $5$-cycle $K$, there is only one vertex that is not a pivot of $K$ in any of them.
Moreover, for every vertex $u\in V(H)$ of degree three and any pair of its neighbors,
there exists $i\in \{1,\ldots,4\}$ such that the vertices of the pair are assigned the same colour by $\varphi_i$
and the other neighbor of $u$ is assigned a different colour.
\end{lemma}
\begin{proof}
Let $H=T'_n$ and let $v_1v_2v_3v_4$ be the $4$-cycle bounding the outer face of $G$, where $v_1v_3$ is an interface
pair of $H$; we will prove the claim by induction on $n$, showing moreover that in
two of the colourings the vertices $v_1$ and $v_3$ receive different colours, and in the other two the vertices
$v_2$ and $v_4$ receive different colours.  The claim is trivial for $n=1$, hence assume that $n\ge 2$
and that $H$ was obtained from a reduced Thomas-Walls graph $H'=T'_{n-1}$ with the outer face bounded by $4$-cycle $v_2v'_3v'_4v'_1$
with the interface pair $v_2v'_4$ by adding the $4$-cycle $v_1v_2v_3v_4$ and the edge $v'_4v_4$.

Let $\varphi_1$, \ldots, $\varphi_4$ be the $3$-colourings of $H'$ obtained by the induction hypothesis, where $\varphi_i(v_2)\neq \varphi_i(v'_4)$
for $i\in \{1,2\}$ and $\varphi_i(v'_1)\neq \varphi_i(v'_3)$ for $i\in\{3,4\}$.  We extend the colourings to $H$ as follows:
\begin{itemize}
\item In $\varphi_1$, we set $\varphi_1(v_4)\colonequals\varphi_1(v_2)$, $\varphi_1(v_1)\colonequals\varphi_1(v'_4)$, and $\varphi_1(v_3)\colonequals\varphi_1(v'_1)$.
\item In $\varphi_2$, we set $\varphi_2(v_4)\colonequals\varphi_2(v_2)$, $\varphi_2(v_1)\colonequals\varphi_2(v'_1)$, and $\varphi_2(v_3)\colonequals\varphi_2(v'_4)$.
\item In $\varphi_3$, we set $\varphi_3(v_4)\colonequals\varphi_3(v'_1)$ and $\varphi_3(v_1)=\varphi_3(v_3)\colonequals\varphi_3(v'_3)$.
\item In $\varphi_4$, we set $\varphi_4(v_4)\colonequals\varphi_4(v'_3)$ and $\varphi_4(v_1)=\varphi_4(v_3)\colonequals\varphi_4(v'_1)$.
\end{itemize}
Observe that the colourings satisfy the required properties, with only $v_2$ not being the pivot of the newly added $5$-faces
$v'_1v'_4v_4v_1v_2$ and $v'_3v'_4v_4v_3v_2$ in any of the $3$-colourings.  For the second part of the claim regarding
neighborhoods of vertices of degree three, it suffices to check neighborhoods of the vertices $v_4$ and $v'_4$.
\end{proof}

A \emph{patch with interface vertices $a$, $b$, and $c$} is a plane graph $F$ with the outer face bounded by an induced cycle $C$ of of length $6$,
where $a$, $b$, and $c$ are distinct non-adjacent vertices of $C$, such that every face of $F$ other than the one bounded by $C$ has length $4$
and no vertex of $G$ is adjacent to all of $a$, $b$, and $c$.
Let $G$ be a plane graph.  Let $G'$ be any graph which can be obtained from $G$ as follows.
Let $S$ be an independent set in $G$ such that every vertex of $S$ has degree $3$.  For each vertex $v\in S$
with neighbors $a$, $b$, and $c$, remove $v$, add new vertices $a'$, $b'$, and $c'$, and a $6$-cycle
$C=aa'bb'cc'$, and draw any patch with interface vertices $a$, $b$, and $c$ in the disk bounded by $C$.
We say that any such graph $G'$ is obtained from $G$ by \emph{patching}.
This operation was introduced by Borodin et al.~\cite{4c4t} in the context of describing planar $4$-critical graphs with exactly $4$ triangles.

Consider a reduced Thomas-Walls graph $G=T'_n$ for some $n\ge 1$, with interface pairs $u_1u_3$ and $v_1v_3$.
A \emph{patched Thomas-Walls graph} is any graph obtained from such a graph $G$ by patching, and $u_1u_3$ and $v_1v_3$ are
its interface pairs (note that $u_1$, $u_3$, $v_1$, and $v_3$ have degree two in $G$, and thus they are not affected by patching).

We start by proving a variant of Lemma~\ref{lemma-mono} for supergraphs of patches.
\begin{lemma}\label{patches}
Suppose $G$ is a triangle-free plane graph with the outer face bounded by a $6$-cycle $K=v_1v_2\ldots v_6$,
with a subgraph $H$ containing $K$ such that every other face of $H$ has length $4$.  Let $S=\{v_1,v_3,v_5\}$.
If $V(G)\neq V(K)$ and no vertex of $G$ is adjacent to all vertices of $S$, then there exists a colouring $\varphi$ of $G$ by subsets of $[3]$ such that a vertex $z\in V(G)\setminus V(K)$
is assigned a set of size two, all other vertices of $G$ are assigned a set of size $1$, and $|\bigcup_{v\in S} \varphi(v)|=2$.
\end{lemma}
\begin{proof}
Note that $H$ is a bipartite graph; let $\varphi_0$ be its $2$-colouring.  If $G=H$, then let $z\in V(G)\setminus V(K)$ be an arbitrary vertex;
since no vertex of $G$ is adjacent to all vertices of $S$, by symmetry assume that $v_1z\not\in E(G)$.  We let $\varphi$ be obtained from $\varphi_0$
by giving $v_1$ the colour $3$ and $z$ the set $\{\varphi_0(z),3\}$.

Otherwise, let $C$ be a $4$-cycle bounding a full $4$-face of $H$
and let $z$ be a vertex of $G$ of degree at most three
drawn in the open disk bounded by $C$ (which exists by Observation~\ref{obs-three}).  Let $\psi$ be a colouring of the subgraph of $G$
drawn in the closed disk bounded by $K$ obtained by Lemma~\ref{lemma-mono}, such that $|\psi(z)|=2$.  Permute the colours in $\psi$
so that $\psi$ matches $\varphi_0$ on all but at most one vertex $w$ of $C$; if $\psi$ matches $\varphi_0$ on $C$, then let $w\in V(C)$ be arbitrary.

Since $\varphi_0(v_1)=\varphi_0(v_3)=\varphi_0(v_5)$, either $w$ has no neighbor in $S$, or every vertex in $S\cap V(C)$ is a neighbor of $w$.
Since no vertex of $G$ is adjacent to all vertices of $S$, in either case we conclude that there exists a vertex $s\in S\setminus V(C)$ non-adjacent
to $w$.  Let $\varphi$ be obtained from $\varphi_0$ by changing the colour of $w$ to $\psi(w)$, colouring all vertices in the open disk bounded by $C$ according to $\psi$,
changing the colour of $s$ to $3$, and extending the colouring to $G$ by Theorem~\ref{thm-extend}.  Observe that $\varphi$ has the required properties.
\end{proof}

Next, let us deal with full $5$-faces of a patched Thomas-Walls graph.
It would be convenient to find a small number of $3$-colourings
of a reduced Thomas-Walls graph such that each vertex of a $5$-face is a pivot in at least one of them; then, we could
use Lemma~\ref{lemma-mono} to gain a bit to size of the independent set for each full $5$-face.  Unfortunately, this
is not possible, and we will have to make do with the weaker claim of Lemma~\ref{lemma-coltw}.
Hence, we need to obtain a variant of Lemma~\ref{lemma-mono} ensuring that a given vertex of a $5$-cycle is not the pivot.

Let $K=v_1v_2v_3v_4v_5$ be a $5$-cycle. We say the pairs
$v_2v_4$ and $v_3v_5$ are \emph{$v_1$-pivot preventing chords}---note that in a $3$-colouring of $K$ together with one
of these chords, $v_1$ cannot be the pivot.  Suppose $G$ is a plane triangle-free
graph with the outer face bounded by a $5$-cycle $K$ and let $x$ be a vertex of $K$.  A colouring $\varphi$ of $G$ by subsets
of $[3]$ is \emph{$v$-nice} if there exists a vertex $z\in V(G)\setminus V(K)$ such that $|\varphi(v)|=1$ for all $v\in V(G)\setminus \{z\}$, $|\varphi(z)|=2$,
and $x$ is not the pivot of the restriction of $\varphi$ to $K$.

\begin{lemma}\label{cyclelemma}
Suppose $G$ is a plane triangle-free graph with the outer face bounded by a $5$-cycle $K=v_1v_2\ldots v_5$.
If $G\neq K$, then for every $k\in\{1,\ldots,5\}$, there exists a $v_k$-nice colouring of $G$.
\end{lemma}
\begin{proof}
Suppose for a contradiction that $G$ is a counterexample with the smallest number of vertices.
Clearly, $G$ has minimum degree at least two, as otherwise we can extend any $3$-colouring of $K$ in that $v_k$
is not the pivot by Theorem~\ref{thm-extend} and give an additional colour to a vertex of degree at most $1$.

We claim that no vertex $v\in V(G)\setminus V(K)$ has two neighbors in $K$.  Indeed, suppose
that $v$ is adjacent say to $v_1$ and $v_3$, and let $K_1=v_1v_2v_3v$ and $K_2=v_3v_4v_5v_1v$.
If $K_2$ does not bound a face, then by the minimality of $G$, the subgraph of $G$ drawn in the closed disk bounded by $K_2$ has
a $v'_k$-nice colouring, where $v'_k=v_k$ if $k\neq 2$ and $v'_k=v$ if $k=2$.
This colouring can be extended to a $v_k$-nice colouring of $G$ by giving $v_2$ the colour of $v$ and
applying Theorem~\ref{thm-extend} to the subgraph of $G$ drawn in the open disk bounded by $K_1$.
If $K_2$ bounds a face, then let $z$ be a vertex of degree at most three in $V(G)\setminus V(K)$
(which exists by Observation~\ref{obs-three}).  We colour $G$ by Lemma~\ref{lemma-mono} so that $z$ is assigned a set of size $2$,
and recolour vertices $v_4$ and $v_5$ so that $v_k$ is not a pivot of the restriction of the colouring to $K$ if necessary.

Since $G$ has minimum degree at least two and no vertex in $V(G)\setminus V(K)$ has more than one neighbor in $K$,
it follows that $|G|\ge 7$.  Suppose that $G$ contains a $4$-face $u_1u_2u_3u_4$.  Since $G$ is plane and triangle-free,
it cannot contain both a path of length three between $u_1$ and $u_3$ and a path of length three between $u_2$ and $u_4$.
By symmetry, assume that there is no path of length three between $u_1$ and $u_3$.  Since no vertex in $V(G)\setminus V(K)$ has more than one neighbor in $K$,
at most one of $u_1$ and $u_3$ belongs to $K$.  Let $G'$ be the graph obtained from $G$ by identifying $u_1$ with $u_3$ to a new vertex $u$
and suppressing parallel edges.  Note that $G'$ is triangle-free and $|G'|=|G|-1\ge 6$, and thus $G'$ has a $v_k$-nice
colouring.  By giving $u_1$ the colour set of $u$ and $u_3$ a single colour from this set, we obtain a $v_k$-nice colouring of $G$.
This contradiction shows that $G$ has no $4$-faces.

Suppose that $G$ contains a $4$-cycle $K'$.  By the previous paragraph, $K'$ does not bound a face.
Let $v$ be a vertex of $G$ of degree at most $3$ contained in the open disk $\Lambda$ bounded by $K'$, which exists
by Observation~\ref{obs-three}.  Let $\varphi'$ be a colouring
obtained by applying Lemma~\ref{lemma-mono} for the subgraph of $G$ drawn in the closure of $\Lambda$ and the vertex
$v$.  Let $G'$ be the graph obtained from the subgraph of $G$ drawn in the complement of $\Lambda$ by adding
a $v_k$-pivot preventing chord.  Since no vertex in $V(G)\setminus V(K)$ has more than one neighbor in $K$, $G'$ contains
exactly one triangle, and thus the colouring of $K'$ given by $\varphi'$ extends to a $3$-colouring $\varphi''$ of $G'$ by Theorem~\ref{thm-extend43}.  The combination
of $\varphi'$ and $\varphi''$ is a $v_k$-nice colouring of $G$, which is a contradiction.  Hence, $G$ contains no $4$-cycles.

Consider any vertex $z\in V(G)\setminus V(K)$ of degree $t\le 3$, and let $z_1$, \ldots, $z_t$ be the neighbors of $z$.
Let $G''$ be the graph obtained from $G-z$ by adding a $v_k$-pivot preventing chord $e_0$ and a $2t$-cycle $K'=z_1z'_1z_2z'_2\ldots z_tz'_t$
with new vertices $z'_1$, \ldots, $z'_t$, redrawn so that $K'$ bounds the outer face of $G''$.  Note that $G''$ contains only one triangle $T$
and that if $t=3$, then every $4$-cycle in $G''$ contains the edge $e_0$.  Observe that if $G''$ has a $3$-colouring such that $z_1$, \ldots, $z_t$
have colour $1$, we can transform it into a $v_k$-nice colouring of $G$ by assigning $v$ the set $\{2,3\}$.  Otherwise,
Theorem~\ref{thm-extend43} and Corollary~\ref{cor-extend63} imply that $t=3$ (i.e., all vertices in $V(G)\setminus V(K)$ have degree at least three)
and $T$ either contains a vertex of $K'$, or all vertices of $T$ have a neighbor in $K'$.

From now on, we can by symmetry assume that $k=1$.  By Lemma~\ref{lemma-three-no2}, there exists a vertex $z\in V(G)\setminus V(K)$ of degree three that has no neighbor in $\{v_3,v_4\}$.
Since $z$ has at most one neighbor in $K$, by symmetry we can assume that $zv_2\not\in E(G)$.
By the previous paragraph applied with $e_0=v_2v_4$, we conclude that $z$ has common neighbors $z_1$, $z_2$, and $z_3$ with $v_2$, $v_3$, and $v_4$, respectively, 
necessarily distinct since $G$ is triangle-free.  Since $z_2$ has degree at least three, Lemma~\ref{lemma-three-no2} implies that one of the open disks bounded
by $5$-cycles $zz_1v_2v_3z_2$ and $zz_3v_4v_3z_2$ contains a vertex $z'$ of degree three with no neighbor in $\{v_3,v_4\}$.  Note that $z'$ cannot have a common
neighbor with both $v_2$ and $v_4$.  Applying the previous paragraph again with $e_0=v_2v_4$, we conclude that $z'$ is adjacent to $v_2$.  Then $z'$ is
non-adjacent to $v_5$ and has no common neighbor with $v_4$, and applying the previous paragraph with $e_0=v_3v_5$, we obtain a contradiction.
\end{proof}

We now combine these results as usual.

\begin{lemma}\label{lemma-full5}
Suppose a triangle-free plane graph $G$ contains a patched Thomas-Walls graph $H$ with $p$ patches and $q$ full $5$-faces as a subgraph.
Then $G$ has an independent set of size at least $\frac{|G|+(p+q)/4}{3}$.
\end{lemma}
\begin{proof}
Let $H'$ be the subgraph of $H$ obtained by removing the internal vertices of the patches,
keeping just their boundary $6$-cycles.  Let $H''$ be the reduced Thomas-Walls graph
obtained from $H'$ by identifying the vertices $a_P$, $b_P$, and $c_P$ of degree two in each boundary 6-cycle of a patch $P$
to a single vertex $v_P$.  The $3$-colourings of $H''$ obtained by Lemma~\ref{lemma-coltw}
can be naturally transformed into $3$-colourings $\varphi'_1$, \ldots, $\varphi'_4$ of $H'$
by giving $a_P$, $b_P$, and $c_P$ the colour of $v_P$ for each patch $P$.
By Theorem~\ref{thm-extend} and Corollary~\ref{cor-patches}, for $i\in\{1,\ldots,4\}$,
the colouring $\varphi'_i$ extends to a colouring $\varphi_i$ of $G$ by subsets of $[3]$ such that each vertex is assigned a non-empty
set.  Furthermore, Lemmas~\ref{patches} and \ref{cyclelemma} together with the properties of $\varphi'_1$, \ldots, $\varphi'_4$
according to Lemma~\ref{lemma-coltw} ensure that we can choose $\varphi_1$, \ldots, $\varphi_4$
so that for each patch $P$ of $H$, there exists a vertex $z\in V(P)$ not contained in the boundary cycle
of the patch such that $|\varphi_j(z)|=2$ for at least one $j\in \{1,\ldots,4\}$; and
for each full $5$-face $f$ of $H$, there
exists a vertex $z\in V(G)\setminus V(H)$ drawn in $f$ such that $|\varphi_j(z)|=2$ for at least one $j\in \{1,\ldots,4\}$.

We replace colours in $\varphi_i$ by $\{3i-2,3i-1,3i\}$ for $i\in\{1,\ldots,4\}$.
Together, the four colourings give a colouring of $G$ by subsets of $[12]$ such that
each vertex is assigned a set of size at least $4$ and each patch contains a vertex
assigned a set of size $5$.  By Lemma~\ref{lemma-indset},
$$\alpha(G)\ge \frac{4|G|+p+q}{12}=\frac{|G|+(p+q)/4}{3}.$$
\end{proof}

\section{Large independent sets}\label{sec-main}

We now proceed with the proof of Theorem~\ref{thm-main}.  The basic idea is to find a large
number of vertices of $G$ that are far apart, and find a colouring by non-empty subsets of $[3]$ that
gives these vertices two colours.
Of course, such a set does not necessarily have to exist (e.g., in a star, all vertices are distance
at most two from one another).  However, this can be worked around by removing a bounded number
of vertices first.
\begin{theorem}[Ne\v{s}et\v{r}il and Ossona de Mendez~\cite{npom-nd1}]\label{thm-wide}
For all integers $s$ and $d$, there exists $n_0$ as follows.
Let $G$ be a planar graph and let $S'$ be a set of its vertices.  If $|S'|\ge n_0$,
then there exist disjoint sets $S\subset S'$ and $Z\subset V(G)$ such that $|Z|\le 2$,
$|S|=s$, and the distance between any two vertices of $S$ in $G-Z$ is at least $d$.
\end{theorem}
By Euler's formula, planar triangle-free graphs have average degree less than $4$,
and thus they have many vertices of degree $4$.
\begin{observation}\label{obs-deg4}
Every planar triangle-free graph $G$ has at least $|G|/5$ vertices of degree at most $4$.
\end{observation}
In view of Lemma~\ref{lemma-mono}, these results give us a hope that the strategy could succeed.

Of course, it may be the case that the described colouring does not exist.
To deal with this issue, we will need some deep results about precolouring extension
in plane triangle-free graphs. The following is an easy consequence of Lemma~5.2 of~\cite{trfree6}.

\begin{theorem}[Dvo\v{r}\'ak et al.~\cite{trfree6}]\label{thm-weight}
There exists a constant $\beta\ge 1$ as follows.  Let $G$ be a plane triangle-free graph
in that every $4$-cycle bounds a face, and let $S$ be a set of vertices of $G$.
There exists a subgraph $F$ of $G$ such that $S\subseteq V(F)$, for every face $f$ of $F$ each $3$-colouring
of the boundary of $f$ extends to a $3$-colouring of the subgraph of $G$ drawn in the closure of $f$,
and $F$ contains at most $\beta|S|$ vertices such that not all incident faces of $F$ are $4$-faces.
\end{theorem}

We need the another result to deal with the non-facial $4$-cycles.
Let $G$ be a plane graph and let $H$ be its subgraph.  We say that a face $f$ of $H$ is \emph{TW-full}
if $f$ has two boundary cycles $C_1$ and $C_2$, both of length $4$, and $G$ contains a subgraph
isomorphic to a patched Thomas-Walls graph whose $4$-faces are bounded by $C_1$ and $C_2$, drawn in the closure
of $f$.  We say that the face is \emph{$\ell$-TW-full} if this subgraph is obtained by patching
from a reduced Thomas-Walls graph $T'_n$ with $n\ge \ell$.

\begin{theorem}[Dvo\v{r}\'ak and Lidick\'{y}~\cite{cylgen-part2}]\label{thm-cylinder}
There exists a constant $\delta\ge 5$ as follows.  Let $G$ be a plane triangle-free graph
and let $C_1$ and $C_2$ be cycles bounding two distinct $4$-faces of $G$.  If the distance
between $C_1$ and $C_2$ in $G$ is at least $\delta$ and some $3$-colouring of $C_1\cup C_2$
does not extend to a $3$-colouring of $G$, then $G$ has a subgraph $H$ with $C_1$ and $C_2$ contained in different components of $H$
such that $|H|\le 12$, a face of $H$ is TW-full and all other faces of $H$ are $4$-faces.
\end{theorem}

A family $\FF$ of cycles in a plane graph $G$ is \emph{laminar} if for any cycles $K_1,K_2\in \FF$,
the open disks bounded by $K_1$ and $K_2$ either are disjoint, or one of them is a subset of the other one.
We define $R_\FF$ as the rooted tree whose vertices are subsets of the plane, defined as follows:
the root is the whole plane, and for each vertex $\Lambda$ of $R_\FF$,
the sons of $\Lambda$ in $R_\FF$ are the inclusion-wise maximal open disks bounded by cycles of $\FF$ that are properly contained
in $\Lambda$.  For each $\Lambda\in V(R_\FF)$ with sons $\Lambda_1$, \ldots, $\Lambda_p$, let $\Lambda^\circ=\Lambda\setminus\bigcup_{i=1}^p \Lambda_p$,
and let $G_\Lambda$ denote the subgraph of $G$ drawn in the closure of $\Lambda^\circ$.  Let $D_\Lambda$ be the set of vertices of the cycles
of $\FF$ forming the boundaries of $\Lambda_1$, \ldots, and $\Lambda_p$. If $\Lambda$ is not the root of $R_\FF$, then let $U_\Lambda$ denote the
set of vertices of the cycle bounding $\Lambda$, otherwise let $U_\Lambda=\emptyset$.

\begin{corollary}\label{cor-weight}
For every positive integer $\ell$, there exists a constant $\beta_1$ as follows.  Let $G$ be a plane triangle-free graph and let $S$ be a non-empty
set of vertices of $G$.
There exists a subgraph $F$ of $G$ such that $S\subseteq V(F)$, every face $f$ of $F$ is either $\ell$-TW-full or
each $3$-colouring of the boundary of $f$ extends to a $3$-colouring of the subgraph of $G$ drawn in the closure of $f$,
and $F$ contains at most $\beta_1|S|$ vertices such that not all incident faces of $F$ are $4$-faces.
\end{corollary}
\begin{proof}
Let $\beta$ be the constant of Theorem~\ref{thm-weight} and $\delta$ the constant of Theorem~\ref{thm-cylinder}.
We set $\beta_1=(4\max(\delta,\ell)+33)\beta$.

If $G$ contained a non-facial $4$-cycle $K$ that does not separate any two vertices of $S$, we can remove the vertices and edges
drawn in the part of the plane minus $K$ not containing $S$, since any $3$-colouring extends to the removed part by Theorem~\ref{thm-extend}.
Hence, assume that every non-facial $4$-cycle in $G$ separates a pair of vertices of $S$.

Let $\FF$ be a maximal laminar family of non-facial $4$-cycles of $G$, and consider the tree $R_\FF$.  For every $\Lambda\in V(R_\FF)$,
the maximality of $\FF$ implies that $G_\Lambda$ does not have any non-facial $4$-cycle.  Furthermore, since every non-facial $4$-cycle
in $G$ separates a pair of vertices of $S$, if $\Lambda$ is a leaf of $R_\FF$, then a vertex of $S$ is drawn in $\Lambda$, and in
particular $R_\FF$ has at most $|S|$ leaves.  Let $X$ be the set of vertices $\Lambda\in V(R_\FF)$ such that either $\Lambda$ has
more than one son in $R_\FF$, or a vertex of $S$ is contained in $\Lambda^\circ$; we have $|X|<2|S|$.  Note that each leaf of $R_\FF$
as well as the root of $R_\FF$ belong to $X$.  For $\Lambda\in X$,
let $S_\Lambda$ be the set consisting of $U_\Lambda\cup D_\Lambda$ and of the vertices of $S$ contained in $\Lambda^\circ$.
Note that for every son $\Lambda'$ of $\Lambda$, there exists a descendant of $\Lambda'$ belonging to $X$,
and thus $\sum_{\Lambda\in X}|D_\Lambda|<4|X|$.  Trivially, we have $\sum_{\Lambda\in X}|U_\Lambda|=4|X|$. Consequently, $\sum_{\Lambda\in X} |S_\Lambda|< 8|X|+|S|\le 17|S|$.
For each $\Lambda\in X$, let $F_\Lambda$ be the subgraph of $G_\Lambda$ obtained by applying Theorem~\ref{thm-weight} for $G_\Lambda$
and $S_\Lambda$.

Note that $R_\FF-X$ is a union of at most $|X|$ paths.  For each such path $P$ whose vertex closest to the root is $\Lambda_1$
and the opposite endvertex is $\Lambda_2$, let $G_P=\bigcup_{\Lambda\in V(P)} G_\Lambda$ and $C_P=C_1\cup C_2$, where
$C_1$ is the cycle bounding $\Lambda_1$ and $C_2$ is the cycle bounding the son of $\Lambda_2$.  Note that vertices of $S$ in $G_P$ may
only belong to $C_1$.  Let $Y$ be the set of paths $P$ such that
the distance between $C_1$ and $C_2$ in $G_P$ is less than $\max(\delta,\ell)$, and let $Z$ be the set of paths $P$ such that 
the distance between $C_1$ and $C_2$ in $G_P$ is at least $\max(\delta,\ell)$.

For $P\in Y$, let $Q$ be a path of length less than $\max(\delta,\ell)$
between $C_1$ and $C_2$, and let $G'_P$ be the graph obtained from $G_P$ by cutting along $Q$; i.e., each vertex of $Q$ gives rise to two
vertices of $G'_P$, each incident with the neighbors on one side of $Q$.  Let $S_P$ consist of $V(C_P)$ and of the vertices of the
two paths of $G'_P$ corresponding to $Q$.  Note that $G'_P$ contains no non-facial $4$-cycles, and let $F'_P$ be the subgraph of $G'_P$
obtained by applying Theorem~\ref{thm-weight} for $G'_P$ and $S_P$.  Let $F_P$ be the subgraph of $G_P$ obtained from $F'_P$ by gluing
back the vertices created by cutting along $P$.

For $P\in Z$, if every $3$-colouring of $C_P$ extends to a $3$-colouring of $G_P$, then let $H_P=C_P$; otherwise,
let $H_P$ be the subgraph of $G_P$ obtained by applying Theorem~\ref{thm-cylinder}.

We let $F=\bigcup_{\Lambda\in X} F_\Lambda\cup\bigcup_{P\in Y} F_P\cup\bigcup_{P\in Z} H_P$.
If a face $f$ of $F$ is a face of $F_z$ for some $z\in X\cup Y$,
then every $3$-colouring of the boundary of $f$ extends to a $3$-colouring of the subgraph of $G$ drawn in the closure of $f$.
If $f$ is a face of $H_P$ for $P\in Z$, then either $f$ is $\ell$-TW-full, or it is a $4$-face and
every $3$-colouring of its boundary extends to a $3$-colouring of the subgraph of $G$ drawn in the closure of $f$
by Theorem~\ref{thm-extend}.

Hence, it suffices to bound the number of vertices of $F$ such that not all incident faces are $4$-faces.
For $\Lambda\in X$, at most $\beta|S_\Lambda|$ vertices are incident with a non-$4$-face of $F_\Lambda$,
and thus there are at most $17\beta|S|$ such vertices in total.
For $P\in Y$, at most $\beta(2\max(\delta,\ell)+8)$ vertices are incident with a non-$4$-face of $F_P$ (by
considering the corresponding faces of $F'_P$), and for $P\in Z$, the graph $H_P$ has at most $12$ vertices.
This gives at most $\beta(2\max(\delta,\ell)+8)|X|\le 2\beta(2\max(\delta,\ell)+8)|S|$ vertices incident with non-$4$-faces
over all paths of $Y\cup Z$.  Consequently, the number of vertices of $F$ incident with non-$4$-faces
is at most $(4\max(\delta,\ell)+33)\beta|S|=\beta_1|S|$ as required.
\end{proof}

We now prove a weaker variant of Theorem~\ref{thm-main}, where the excess of the size of the independent
set over the third of the number of vertices is not linear.

\begin{lemma}\label{lemma-large-fst}
There exists a function $f:\mathbf{N}^2\to\mathbf{N}$ as follows.
Let $k$ and $a$ be positive integers and let $G$ be a plane triangle-free graph not containing a clean Thomas-Walls $k$-tube.
If $|G|\ge f(k,a)$, then $\alpha(G)\ge\frac{|G|+a}{3}$.
\end{lemma}
\begin{proof}
Let $\ell=3(4a+9)k$ and $t=\lceil(5a+10)\beta_1\rceil$, where $\beta_1$ is the constant of Corollary~\ref{cor-weight}.
Let $n_0$ be the constant of Theorem~\ref{thm-wide} applied with $s=a+2$ and $d=96t(3t+a+2)+2$.
Let $f(k,a)=5n_0$.

Let $G$ be a plane triangle-free graph with at least $5n_0$ vertices.  By Observation~\ref{obs-deg4}, $G$
has at least $n_0$ vertices of degree at most $4$. By Theorem~\ref{thm-wide}, there exists $Z\subseteq V(G)$
of size at most two and a set $S_0$ of $a+2$ vertices of degree at most $4$ such that the distance between any
pair of vertices of $S_0$ in $G-Z$ is at least $d$.  Note that the graph $G-Z$ does not contain
a clean Thomas-Walls $3k$-tube, since vertices of $Z$ may make only two faces of such a tube full in $G$
and $G$ does not contain a clean Thomas-Walls $k$-tube.

Let $S$ be the set consisting of $S_0$ and all the neighbors of vertices of $S_0$ in $G-Z$; we have $|S|\le 5|S_0|=5a+10$.
Let $F$ be the subgraph obtained by Corollary~\ref{cor-weight} applied to $G-Z$ and $S$,
such that each face $f$ of $F$ is either $\ell$-TW-full or
each $3$-colouring of the boundary of $f$ extends to a $3$-colouring of the subgraph of $G$ drawn in the closure of $f$,
and such that $F$ contains at most $t$ vertices with not all incident faces of $F$ being $4$-faces.

Suppose first that
$F$ has an $\ell$-TW-full face, and let $H$ be the corresponding patched Thomas-Walls subgraph.
Let $p$ be the number of patches and $q$ the number of full $5$-faces of $H$.
Since $G-Z$ does not contain a clean Thomas-Walls $3k$-tube, we have $3(p+q+1)k\ge \ell=3(4a+9)k$, and thus
$p+q\ge 4a+8$.  Therefore,
$$\alpha(G)\ge\alpha(G-Z)\ge \frac{|G-Z|+(4a+8)/4}{3}=\frac{|G-Z|+a+2}{3}\ge\frac{|G|+a}{3}$$
by Lemma~\ref{lemma-full5}.

Hence, we can assume that no face of $F$ is $\ell$-TW-full.  Note that at most $t$ vertices of $F$ are incident
with non-$4$-faces, and thus if $|F|\ge 96t(3t+a+2)=d-2$, then Lemma~\ref{lemma-sizeind} implies
$\alpha(G)\ge\alpha(G-Z)\ge\frac{|G-Z|+a+2}{3}\ge\frac{|G|+a}{3}$.  Hence, we can assume
$|F|\le d-3$, and since the distance between any two vertices of $S_0$ is at least $d$, if $u$ and $v$ are
neighbours of distinct vertices of $S_0$, then they do not belong to the same component of $F$.

By Lemma~\ref{lemma-mono} applied to each component of $F$ separately, $F$ has a colouring by subsets of $[3]$
such that vertices of $S_0$ are assigned two colours and all other vertices are assigned one colour.
Since $F$ has no TW-full faces and the neighborhood of vertices of $S_0$ in $G-Z$ is contained in $S$,
this colouring extends to a colouring of $G-Z$ with the same property,
and by Lemma~\ref{lemma-indset}, we have
$\alpha(G)\ge\alpha(G-Z)\ge\frac{|G-Z|+a+2}{3}\ge\frac{|G|+a}{3}$
as required.
\end{proof}

The main result is now proved by combining Lemma~\ref{lemma-large-fst} with Theorem~\ref{thm-largenosep}.

\begin{proof}[Proof of Theorem~\ref{thm-main}]
Let $f$ be the function of Lemma~\ref{lemma-large-fst} and let $\gamma$ be the constant of Theorem~\ref{thm-largenosep}.
Let $\gamma_k=\frac{\gamma}{12(f(k,14)+1)}$.

We prove the claim by induction on the number of vertices of $G$.  The claim follows from Corollary~\ref{cor-add1} if $|G|\le 1/\gamma_k$,
and thus assume that $|G|>1/\gamma_k>f(k,14)$.  Let us say that a non-facial $4$-cycle $C$ in $G$ is \emph{substantial} if at
least $f(k,14)$ vertices of $G$ are drawn in the closed disk bounded by $C$.  If $G$ has a substantial $4$-cycle,
then let $C$ be such a $4$-cycle such that the open disk bounded by $C$ is inclusion-wise minimal, let $G_1$ be the subgraph of $G$
drawn in the closed disk bounded by $C$, and let $G_2$ be the subgraph of $G$ drawn in the complement of the open disk bounded by $C$.
If no $4$-cycle in $G$ is substantial, then let $G_1=G$ and let $G_2$ be the null graph.  Note that neither $G_1$ nor $G_2$
contains a clean Thomas-Walls $k$-tube.

Let $\FF$ be the family of non-facial $4$-cycles in $G_1$ such that the open disks bounded by them are inclusion-wise maximal;
note that the disks are pairwise disjoint.  For $K\in\FF$, let $G_K$ denote the subgraph of $G_1$ drawn in the closed disk bounded
by $K$.  By the choice of $G_1$, the graph $G_K$ has less than $f(k,14)$ vertices.  Let $G'_1$ be the subgraph of $G_1$ obtained
by removing the vertices and edges drawn in the open disks bounded by the cycles of $\FF$.  Note that by the choice of $\FF$,
the graph $G'_1$ has no non-facial $4$-cycles.  Furthermore, $|E(G'_1)|\ge |E(\bigcup\FF)|\ge 2|\FF|$, and since $G'_1$ is a simple triangle-free
planar graph, we have $|G'_1|\ge |\FF|$.  Note also that $|G'_1|\ge |G_1|-f(k,14)|\FF|$.  We conclude that
$|G'_1|\ge |G_1|/(f(k,14)+1)$.

By Lemma~\ref{lemma-large-fst} and Theorem~\ref{thm-largenosep}, we have
\begin{align*}
\alpha(G_1)&\ge \frac{|G_1|+\max(14,\gamma|G'_1|)}{3}\ge \frac{|G_1|+\max(14,\gamma|G_1|/(f(k,14)+1))}{3}\\
&=\frac{|G_1|+\max(14,12\gamma_k|G_1|)}{3}=\frac{|G_1|+\max(2,12(\gamma_k|G_1|-1))}{3}+4.
\end{align*}
Consequently,
$$\alpha(G_1)\ge \frac{(1+\gamma_k)|G_1|}{3}+4.$$
By the induction hypothesis, we have $$\alpha(G_2)\ge \frac{(1+\gamma_k)|G_2|}{3}.$$
Removing from the maximum independent sets of $G_1$ and $G_2$ the vertices contained in $C$ (there are at most two such vertices
in each independent set), we obtain
$$\alpha(G)\ge\alpha(G_1)+\alpha(G_2)-4\ge\frac{(1+\gamma_k)|G|}{3}$$
as required.
\end{proof}

\bibliographystyle{acm}
\bibliography{indepkernel}

\end{document}